\documentclass[11pt,a4paper]{amsart}
\usepackage{amssymb,amsmath,amsthm}
\usepackage{extarrows}

\newcommand{\R}{\mathbb{R}}
\newcommand{\s}{\mathbb{S}}
\newcommand{\C}{\mathcal{C}}

\newtheorem{theorem}{Theorem}[section]
\newtheorem{lemma}[theorem]{Lemma}

\newtheorem{remark}{Remark}[section]
\theoremstyle{definition}

\numberwithin{equation}{section}

\title{ Infinitely many solutions for Schr\"{o}dinger-Newton equations}

\author[Y. Hu]{Yeyao Hu} 
\address{Yeyao Hu, School of Mathematics and Statistics, The Central South University,
Changsha, Hunan 410083, P. R. China}
\email{huyeyao@csu.edu.cn}

\author[A. Jevnikar]{ Aleks Jevnikar}
\address{Aleks Jevnikar, Department of Mathematics, Computer Science and Physics, University of Udine, Via delle Scienze 206, 33100 Udine, Italy}
\email{aleks.jevnikar@uniud.it}

\author[W. Xie]{Weihong Xie}
\address{Weihong Xie, School of Mathematics and Statistics, The Central South University,
Changsha, Hunan 410083, P. R. China}
\email{xieweihong0218@163.com}

\begin{document}

 \begin{abstract}
We prove the existence of  infinitely many  non-radial positive solutions for  the  Schr\"{o}dinger-Newton  system
 \begin{equation*}
   \begin{cases}
  \Delta  u- V(|x|)u + \Psi  u=0,\quad &x\in\R^3,\\
  \Delta  \Psi+\frac12 u^2=0, &x\in\R^3,
\end{cases}
 \end{equation*}
provided that $V(r)$ has the following behavior at infinity:
\begin{equation*}
    V(r)=V_0+\frac{a}{r^m}+O\left(\frac{1}{r^{m+\theta}}\right)
 \quad\mbox{ as  } r\rightarrow\infty,
\end{equation*}
where $\frac12\le m<1$ and $a, V_0, \theta$ are some positive constants. In particular, for any $s$ large we use a reduction method to construct $s-$bump solutions lying on a circle of radius $r\sim (s\log s)^{\frac{1}{1-m}}$. 
\end{abstract}

\date{}

\maketitle


\

\vskip6mm
{\section{Introduction and statement of  results}}
 \setcounter{equation}{0}
 In this paper we consider  the following  Schr\"{o}dinger-Newton  system
\begin{equation}\label{SN0}
\begin{cases}
  \Delta  u- V(x)u + \Psi  u=0,\quad &x\in\R^3,\\
  \Delta  \Psi+\frac12 u^2=0, &x\in\R^3.
\end{cases}
\end{equation}
Here $V$  is a given external  potential and $\Psi$ is the Newtonian gravitational potential. The latter model was proposed in  \cite{Penrose}, where  the wave function $u$  represents a stationary  solution for a quantum system describing a  nonlinear modification of the Schr\"{o}dinger equation with a Newtonian gravitational potential representing the interaction of the particle with its own gravitational field.

Note that the second equation of \eqref{SN0} (see \cite{Trudinger1997})  has a unique positive solution $\Psi_{u}\in D^{1,2}(\R^3)$  given by
 \begin{equation}\label{psi}
   \Psi_{u}(x) = \frac{1}{8\pi} \int_{\R^3}\frac{u^2(y)}{|x-y|}dy.
 \end{equation}
Thus,   the system \eqref{SN0} is equivalent  to the following single nonlocal equation:
\begin{equation}\label{SSN}
   -\Delta  u+V(x)u = \frac{1}{8\pi}\left(\int_{\R^3}\frac{u^2(y)}{|x-y|}dy\right) u, \quad x\in\R^3.
\end{equation}
Clearly, $(u,\Psi_u)$ is a solution of the system \eqref{SN0} if and only if $u$ is a solution of the equation \eqref{SSN}. The problem \eqref{SSN} appears also in the study of standing waves of nonlinear Hartree equations, see \cite{erdos, hepp, spohn}. Moreover, it can be seen as a special case of the Choquard equation, see \cite{lieb, Lions0, menzala}.

\
\par
Let us consider the system \eqref{SN0} with $V(x)\equiv1$, that is
\begin{equation}\label{1}
\begin{cases}
  \Delta  u- u + \Psi  u=0,\quad &x\in\R^3,\\
  \Delta  \Psi+\frac12 u^2=0, &x\in\R^3.
\end{cases}
\end{equation}    
The existence of a unique ground state solution to the latter problem has been known since \cite{Lions,Lions2} via variational methods, see also \cite{Tod99} for a more recent result. Moreover, the nondegeneracy of the ground state was proven in \cite{WeiJmp}. We refer to Theorem \ref{nonden} for the precise statements. 

\par
Furthermore, the case $V(x)\not\equiv1$ has been treated in \cite{WeiJmp} in the semi-classical regime, that is the singularly perturbed  Schr\"{o}dinger-Newton  problem
\begin{equation}\label{PSN}
   -\varepsilon^2\Delta  u+V(x)u = \frac{1}{8\pi\varepsilon^2}\left(\int_{\R^3}\frac{u^2(y)}{|x-y|}dy\right) u, \quad x\in\R^3,
\end{equation}
where $\varepsilon>0$ is a  parameter and $\inf_{\R^3}V>0$. The authors prove the existence of positive $K-$bump solutions to \eqref{PSN} concentrating  at local maximum (minimum) or nondegenerate critical points of $V$ as $\varepsilon\to0$. Moreover, it is also shown that there is a  strong interacting between each pair of bumps, which stands in  comparison  with the analogous result for Schr\"{o}dinger equations \cite{Kang}.
It turns out that such positive solutions concentrating at non-degenerate critical points of $V$ are unique for $\varepsilon$  small enough, as recently shown in \cite{Luo} by using local Pohozaev identities.

\
\par
On the other hand, for \eqref{PSN} with $V(x)\not\equiv1$ and $\varepsilon=1$ there are fewer results. In particular, at least to our knowledge, there are no multiplicity results available in the literature. The aim of this paper is to obtain infinitely many  non-radial  solutions to \eqref{SN0} with radial potential $V(r)$, that is,
\begin{equation}\label{SN}
\begin{cases}
  \Delta  u- V(|x|)u + \Psi  u=0,\quad &x\in\R^3,\\
  \Delta  \Psi+\frac12 u^2=0, &x\in\R^3,
\end{cases}
\end{equation}
or equivalently,
\begin{equation}\label{SN1}
   -\Delta  u+V(|x|)u = \frac{1}{8\pi}\left(\int_{\R^3}\frac{u^2(y)}{|x-y|}dy\right) u, \quad x\in\R^3.
\end{equation}
To this end, we make the following assumption on the behavior at infinity of $V$:

\medskip

\begin{itemize}
  \item[($\rm H$)]   There exist some constants $ a,\ \theta,\ V_0>0$  and $\frac12\le m<1$,  such that $V(x)\ge V_0$ and
  \begin{equation}\label{Vr}
    V(r)=V_0+\frac{a}{r^m}+O\left(\frac{1}{r^{m+\theta}}\right) \quad\mbox{ as  } r\rightarrow\infty.
  \end{equation}
\end{itemize}

 \par
 The main result is the following.
 \begin{theorem}\label{main1}
    If $V$ satisfies $(H)$,
then problem \eqref{SN} has infinitely many non-radial positive solutions.
 \end{theorem}

For any $s$ large we will construct solutions with $s$ bumps approaching the infinity. Observe that since we can assume without loss of generality $V_0=1$, the condition $(H)$ yields $$\lim_{r\rightarrow\infty} V(r)=1.$$
  Therefore, we can use the single-peaked ground state solution $U$ of \eqref{1}, see also Theorem \ref{nonden}, as an approximate solution of \eqref{SN}. A combination of single-peaked solutions will give then rise to the $s-$bump solutions we are looking for.
 \par

We center the bumps in 
 \begin{equation}\label{xi}
   \xi_i=\left(r\cos\frac{2(i-1)\pi}{s},r\sin\frac{2(i-1)\pi}{s},0\right),\quad i=1,2,\dots,s,
 \end{equation}
 and we thus let
 \begin{equation*}
   U_r=\sum_{i=1}^s U_{\xi_i}(x),
 \end{equation*}
where $ U_{\xi_i}(x)=U(x-\xi_i)$.  The bumps are lying on a circle for which we choose a radius
\begin{equation}\label{SN-r}
  r\in\left[r_0(s\log s)^{\frac{1}{1-m}},r_1(s\log s)^{\frac{1}{1-m}}\right]
\end{equation}
 for some $r_1>r_0>0$ and $m$ is given in $(H)$. 

\par
Theorem \ref{main1} will be an   immediate consequence of    the following result.
\begin{theorem}\label{main2}
     If $V$ satisfies $(H)$, then   there exists an integer $s_0 > 0$ such that for all $s>s_0$, the problem \eqref{SN} admits  a positive solution  $u_s$ satisfying
\begin{equation*}
  u_s=U_{r_s}+w_s,
\end{equation*}
where $r_s\in\left[r_0(s\log s)^{\frac{1}{1-m}},r_1(s\log s)^{\frac{1}{1-m}}\right]$, $w_s\in  \mathcal E$, $\mathcal E$ is defined in \eqref{E} and
\begin{equation*}
  \int_{\R^3} (|D w_s|^2+w_s^2)\rightarrow0 \quad \mbox{ as }s\rightarrow \infty.
\end{equation*}
\end{theorem}

\

The space $\mathcal{E}$ for the error term is given by
 \begin{equation}\label{E}
 \begin{aligned}
    \mathcal{E}=\bigg\{u\in H^1(\R^3):\, &u \mbox{ is even in }  x_2,\, x_3,\\
   &u(r\cos\theta,r\sin\theta,x_3)=u\left(r\cos\left(\theta+\frac{2\pi i}{s}\right),r\sin\left(\theta+\frac{2\pi i}{s}\right),x_3\right),\\
	 &i=1,2,\dots,s-1\bigg\}.
 \end{aligned}
\end{equation}

We will use a reduction argument to prove the above result. This is quite standard in the singularly perturbed problems alike \eqref{PSN} for $\varepsilon\to0$. Here we exploit the loss of compactness of the domain and use the number  $s$  of the bumps of the solution  $u_s$  as the parameter to  build up the spike solutions.

 \par
 This strategy  was firstly introduced by Wei and Yan  in \cite{WeiCV} and it was later  successful used  to study other  elliptic problems, see for example \cite{WeiDcds,Chen-wy,WeiP,Peng1,Peng2,Wei09,Wang13,WeiJfa,WeiAsna,WeiJmpa}.  The method was originally designed  in \cite{WeiCV}  to prove the existence of  infinitely many solutions for  the Schr\"{o}dinger equation
\begin{equation}\label{S}
   -\Delta  u+V(|x|)u =u^p,\ \ u>0,
\end{equation}
provided $V$ satisfies the above condition $(H)$ with $m>1$ instead of $\frac12\le m<1$. The bump solutions are there modeled on the ground state of the  Schr\"{o}dinger equation \eqref{S} with $V(x)\equiv1$. Compared to this result we have to face here new difficulties due to the non-local nature of the problem. A similar non-local problem, the Schr\"{o}dinger-Possion  system, was considered in \cite{Dsp,Li10}. However, differently from our approach, the bump solutions are there modeled still on the ground state of the  Schr\"{o}dinger equation \eqref{S}. As a matter of fact, in all \cite{WeiCV} and  \cite{Dsp,Li10} the bumps are lying on a circle of radius
$$
    r_s\sim s\log s,
$$
which is essential in the constructions of the spike solutions. On the contrary, in our case, to handle the non-local features and the fact that we build up the bump solutions directly on the ground state of the Schr\"{o}dinger-Newton  system \eqref{1}, we end up constructing spike solutions lying on a circle of radius 
$$
r_s\sim (s\log s)^{\frac{1}{1-m}}.
$$
See  Remark \ref{rem} for further discussion in this respect.

\

\par
The paper is organized as follows. In section \ref{sec:prelim} we set up the problem and estimate the energy of the approximate solution, in section \ref{sec:proof} we give the proof of the main result.


\

\vskip6mm
{\section{ An estimate for the energy of approximate solutions} \label{sec:prelim}} 

In this section we present some preliminaries and give an estimate of the energy of the approximate solution. We first collect some key known results which will be used in the sequel. 
\begin{theorem}\label{nonden} \emph{(\cite{Tod99}, \cite{WeiJmp})}
\par
\noindent
\emph{(Existence)}  There exists a unique radial
 solution $(U,\Psi)$ with
 \begin{equation*}
   U(x)\rightarrow0,\ \ \Psi(x)\rightarrow 0 \quad\mbox{ as }|x|\rightarrow\infty,
 \end{equation*}
 of the following problem
\begin{equation}\label{U-equation}
  \begin{cases}
   \Delta  u-  u + \Psi  u=0,\quad & \mbox{ in }\R^3,\\
  \Delta  \Psi+\frac12 u^2=0, & \mbox{ in }\R^3,\\
   u,\ \Psi>0,\quad u(0)=\max_{x\in\R^3}u(x).
\end{cases}
\end{equation}
Moreover, $U$ is strictly decreasing and
\begin{equation}\label{decay}
  \lim_{|x|\rightarrow \infty}U(x)|x|e^{|x|}=\lambda_0,\quad \lim_{y\rightarrow\infty}\frac{U'(x)}{U(x)}=-1,
\end{equation}
 and
 \begin{equation}\label{decay1}
    \lim_{|x|\rightarrow \infty}\Psi(x)|x| = \lambda_1
 \end{equation}
for some constants $\lambda_0,\ \lambda_1>0$.

\
\par
\noindent
\emph{(Nondegeneracy)}  Suppose that  $\phi\in H^1(\R^3)$ satisfies the following eigenvalue problem:
\begin{equation*}
   -\Delta  \phi+\phi=\frac{1}{8\pi}\left(\int_{\R^3}\frac{U^2(y)}{|x-y|}dy\right) \phi +\frac{1}{4\pi}\left(\int_{\R^3}\frac{U(y)\phi(y)}{|x-y|}dy\right) U.
\end{equation*}
Then
\begin{equation*}
  \phi\in span\left\{\frac{\partial U}{\partial x_i}, i=1,2,3\right\}.
\end{equation*}
\end{theorem}

\

We next introduce the functional-analytic framework.  Define the inner product and norm on the workspace $H^1(\R^3)$ by
  \begin{equation*}
   \langle u,v\rangle=\int_{\R^3}(\nabla u \nabla v+V(|x|)uv)\quad \mbox{and}\quad
     \|u\| =(u,u)^{\frac{1}{2}},
  \end{equation*}
   respectively.  Since $V$ is bounded, the norm $\|\cdot\|$
is equivalent to the standard norm.
   By  Hardy-Littlewood-Sobolev inequality, H\"{o}lder inequality and Sobolev inequality,  for $u,v,w,\phi\in H^1(\R^3)$,
\begin{equation}\label{Hardy-LS}
   \int_{\R^3}\int_{\R^3}\frac{u(y)v(y)w(x)\phi(x)}{|x-y|}\leq \|uv\|_{L^p(\R^3)}\|w\phi\|_{L^q(\R^3)} \leq C \|u\| \|v\| \|w\| \|\phi\|,
 \end{equation}
 where $\frac{1}{p}+\frac{1}{q}=\frac53$. In particular,
\begin{equation}\label{psi-norm}
   ||\Psi_u||_{D^{1,2}(\R^3)}^2=\int_{\R^3}\Psi_u u^2 \leq C\|u\|^4_{L^{\frac{12}5}(\R^3)}\leq C||u||^4.
 \end{equation}
 Throughout this paper,  $C$  denotes various positive constants whose exact value is inessential.
   \par
The  associated energy  functional  to \eqref{SN1} is defined  as follows
\begin{equation}\label{functional}
  J(u)=\frac{1}{2}\int_{\R^3}(|\nabla u|^2 +V(|x|)u^2)-\frac{1}{32\pi}\int_{\R^3}\int_{\R^3}\frac{u^2(y)u^2(x)}{|x-y|}.
\end{equation}
Obviously, $J(u)\in\C^2(H^1(\R^3),\R)$.

 \par
 Recall that $ \xi_i=\left(r\cos\frac{2(i-1)\pi}{s},r\sin\frac{2(i-1)\pi}{s},0\right)$.
 In what follows, we assume that
\begin{equation}\label{r-range}
  r\in I_s:=\left[\left(\left( \frac{A_1}{64am\pi^2}\right)^{\frac{1}{1-m}}-\alpha\right)(s\log s)^{\frac{1}{1-m}},\left(\left( \frac{A_1}{64am\pi^2}\right)^{\frac{1}{1-m}}+\alpha\right)(s\log s)^{\frac{1}{1-m}}\right],
\end{equation}
where $a$ and $m$ is the constants in $(H)$, $\alpha>0$ is a small constant   and
  \begin{equation}\label{A1}
    A_1=\int_{\R^3} U^2.
  \end{equation}

\par
To evaluate $J(U_r)$ we need to estimate the interacting term $\sum_{i=2}^s \int_{\R^3}\Psi_{ U_{\xi_{1}}} U_{\xi_{i}}^2$. The following holds.
\begin{lemma}\label{lemma-sum}
  There is a small constant $\beta>0$ such that
   \begin{equation}\label{interact}
      \sum_{i=2}^s \int_{\R^3}\Psi_{ U_{\xi_{1}}} U_{\xi_{i}}^2=\frac{A_1^2}{32\pi^2}\frac{s\log s}{r}+ O\left( \frac1{ s^{\frac{2m}{1-m}+\beta}}\right).
   \end{equation}
     where $A_1$ is introduced in \eqref{A1}.

\end{lemma}
\begin{proof}
A direct calculation gives that
\begin{equation*}
\begin{aligned}
    \sum_{i=2}^s \int_{\R^3}\Psi_{ U_{\xi_{1}}} U_{\xi_{i}}^2&= \frac{1}{32\pi}\sum_{i=2}^s\int_{\R^3}\int_{\R^3}\frac{U_{\xi_{1}}^2(y)U_{\xi_{i}}^2(x)}{|x-y|}\\
    &=\frac{1}{32\pi}\sum_{i=2}^s\int_{\R^3}U^2(x)\int_{\R^3}U^2(y)\frac{1}{|x-y+(\xi_1-\xi_i)|}\\
    &=\frac{1}{32\pi}\sum_{i=2}^s\frac{1}{|\xi_1-\xi_i|}\int_{\R^3}U^2(x)\int_{\R^3}U^2(y)
    +  O\left(\sum_{i=2}^s\frac{1}{|\xi_1-\xi_i|^2}\right)
\end{aligned}
\end{equation*}

Observe that
\begin{equation}\label{distance}
\begin{aligned}
   \sum_{i=2}^s\frac{1}{|\xi_1-\xi_i|}&=\frac{1}{2r} \sum_{i=1}^{s-1}\frac{1}{ \sin\frac{i\pi}{s} }.
\end{aligned}
 \end{equation}
 It is readily checked that
\begin{equation*}
\int_{\frac32}^{s-\frac{3}{2}}\frac{1}{\sin\frac{x\pi}{s}} \leq \sum_{i=1}^{s-1}\frac{1}{\sin\frac{i\pi}{s}}\leq \int_{\frac12}^{s-\frac{1}{2}}\frac{1}{\sin\frac{x\pi}{s}}
\end{equation*}
and
\begin{equation*}
  \lim_{s\rightarrow\infty}\frac{1}{s\log s}\int_{\frac32}^{s-\frac{3}{2}}\frac{1}{\sin\frac{x\pi}{s}}= \lim_{s\rightarrow\infty}\frac{1}{s\log s}\int_{\frac12}^{s-\frac{1}{2}}\frac{1}{\sin\frac{x\pi}{s}}
  =\frac2\pi.
\end{equation*}
Thus, we have
\begin{equation}\label{sum1i}
  \sum_{i=2}^s\frac{1}{|\xi_1-\xi_i|}=\frac{s \log s}{\pi r}+o_s(1).
\end{equation}
Similarly,
\begin{equation*}
  O\left(\sum_{i=2}^s\frac{1}{|\xi_1-\xi_i|^2}\right)=  O\left( \frac1{ s^{\frac{2m}{1-m}+\beta}}\right).
\end{equation*}
Putting these facts together,  \eqref{interact} follows.
\end{proof}

\

 Next we give the energy estimate for the approximate solution $U_r$.
 \begin{lemma}\label{energy-estimate}
  There is a small constant $\beta>0$ such that
  \begin{equation*}
      J(U_r)= s\left[\frac{A_2}{16\pi}+\frac{aA_1}{2} \frac{1}{r^m}- \frac{A_1^2}{128\pi^2}\frac{s\log s}{r}+  O\left( \frac1{ s^{\frac{2m}{1-m}+\beta}}\right) \right],
  \end{equation*}
  where $A_1$ is defined in \eqref{A1} and
  \begin{equation}\label{A2}
    A_2=\int_{\R^3}\int_{\R^3}\frac{U^2(y)U^2(x)}{|x-y|}.
  \end{equation}
 \end{lemma}
\begin{proof}
\begin{equation*}
  \begin{aligned}
      J(U_r)&=\frac{1}{2}\int_{\R^3}(|\nabla U_r|^2+ U_r^2)+\frac{1}{2}\int_{\R^3}(V(|x|)-1) U_r^2-\frac{1}{32\pi}\int_{\R^3}\int_{\R^3}\frac{U_r^2(y)U_r^2(x)}{|x-y|}\\
      &=I_1+I_2-I_3,
  \end{aligned}
\end{equation*}
where $I_1,\ I_2,\ I_3$ are defined by the last equality.

\par
Note that the property \eqref{decay} of $U$ indicates that
\begin{equation*}
   \int_{\R^3}U_{\xi_1}U_{\xi_i}=O\left( e^{-(1-\delta)|\xi_1-\xi_i|}\right)
\end{equation*}
  for $i\neq 1$ and for any small $\delta>0$. We further have
   \begin{equation}\label{U1U2}
   \sum_{i=2}^s e^{-(1-\delta)|\xi_1-\xi_i|} \leq Ce^{-(1-\delta)\frac{2\pi r}{s}}=O\left( e^{-(1-\delta)s^{\frac m{1-m}}(\log s)^{\frac 1{1-m}}}\right)=o\left(e^{- s^{\frac m{1-m}}}\right).
 \end{equation}
By \eqref{U-equation}, \eqref{psi-norm} and   \eqref{U1U2}, we have
\begin{equation}\label{I1}
  \begin{aligned}
    I_1&=\frac{s}{2}\int_{\R^3}(|\nabla U_{\xi_1}|^2+ U_{\xi_1}^2)+\frac{s}{2}\sum_{i=2}^s \int_{\R^3}(\nabla U_{\xi_1} \nabla U_{\xi_i}+U_{\xi_1}U_{\xi_i} )\\
    &=\frac{s}{16\pi} A_2+\frac{s}{2} \int_{\R^3}\Psi_{U_{\xi_1}}U_{\xi_1}\left(\sum_{i=2}^s U_{\xi_i} \right)\\
    &=\frac{s}{16\pi} A_2 +sO\left(\sum_{i=2}^se^{-(1-\delta)|\xi_1-\xi_i|}\right)\\
    &=\frac{s}{16\pi}A_2+ o\left(se^{-s^{\frac m{1-m}}}\right).
  \end{aligned}
\end{equation}

\par
Following \cite[Proposition A.2]{WeiCV},
 a elementary calculation shows that for any $t>0$,
\begin{equation*}
 \frac{1}{|x-\xi_1|^{t}}=\frac{1}{|\xi_1|^{t}}\left(1+O\left(\frac{|x|}{|\xi_1|}\right)\right), \ x\in B_{\frac{|\xi_1|}{2}}(0)
\end{equation*}
and then
\begin{equation*}
  \begin{aligned}
    \int_{\R^3}(V(|x|)-1) U^2_{\xi_1}&=\int_{B_{\frac r2}(0)}(V(|x-\xi_1|)-1) U^2 +O\left(e^{-(1-\eta)r}\right)\\
    &=\int_{B_{\frac r2}(0)}\left[\frac{a}{|x-\xi_1|^m}+O\left(\frac{1}{|x-\xi_1|^{m+\theta}} \right)\right] U^2 +O\left(e^{-(1-\delta)r}\right)\\
    &=\frac{aA_1}{r^m}+O\left(\frac{1}{r^{m+\theta}} \right).
  \end{aligned}
\end{equation*}
Therefore,
\begin{equation}\label{I2}
\begin{aligned}
    I_2&=\frac{s}{2}\int_{\R^3}(V(|x|)-1) U^2_{\xi_1}+\frac{s}{2}\int_{\R^3}(V(|x|)-1) U_{\xi_1}\left(\sum_{i=2}^s U_{\xi_i} \right)\\
    &=\frac{aA_1s}{2 r^m}+  +sO\left(\sum_{i=2}^se^{-(1-\delta)|\xi_1-\xi_i|}\right)\\
    &=\frac{aA_1s}{2 r^m}+  o\left(se^{-s^{\frac m{1-m}}}\right).
\end{aligned}
\end{equation}

\par
   Using   \eqref{interact} and \eqref{U1U2},  we compute
\begin{equation}\label{I3}
  \begin{aligned}
    I_3&=\frac{1}{4}\int_{\R^3} \Psi_{U_r}\left(s U_{\xi_{i}}^2+\sum_{i\neq j} U_{\xi_{i}}U_{\xi_{j}}\right)=\frac{s}{4}\int_{\R^3}\Psi_{ U_{\xi_{1}}}U_r^2+ o\left(e^{-s^{\frac m{1-m}}}\right)\\
    &=\frac{s}{4}\int_{\R^3}\Psi_{ U_{\xi_{1}}}\left(\sum_{i=1}^s U_{\xi_{i}}^2 +\sum_{i\neq j} U_{\xi_{i}}U_{\xi_{j}}\right) +o\left(e^{-s^{\frac m{1-m}}}\right)\\
    &=\frac{s}{4}\int_{\R^3}\Psi_{ U_{\xi_{1}}} U_{\xi_{1}}^2+\frac{s}{4}\sum_{i=2}^s \int_{\R^3}\Psi_{ U_{\xi_{1}}} U_{\xi_{i}}^2+o\left(e^{-s^{\frac m{1-m}}}\right)\\
    &=\frac{s}{32\pi}A_2+\frac{A_1^2}{128\pi^2}\frac{s^2\log s}{r}+
       O\left( \frac1{ s^{\frac{3m-1}{1-m}+\beta}}\right).
  \end{aligned}
\end{equation}
We conclude from \eqref{I1}--\eqref{I3} that
\begin{equation*}
  J(U_r)= s\left[\frac{A_2}{16\pi}+\frac{aA_1}{2} \frac{1}{r^m}- \frac{A_1^2}{128\pi^2}\frac{s\log s}{r}+  O\left( \frac1{ s^{\frac{2m}{1-m}+\beta}}\right) \right],
\end{equation*}
as desired.
\end{proof}

\

\vskip6mm
{\section{Proof of the main results} \label{sec:proof}} 
 \setcounter{equation}{0}
In this section, we use the reduction method to reduce the construction of spike solutions for \eqref{SN1}  to finding critical points of a certain function $\mathcal F(r)$, $r\in I_s$.
\par
Recall that $\xi_1=(r,0,0)$ and we    assume that
 $$r\in I_s.$$
 Define
 \begin{equation}\label{Es}
   \mathcal E_s=\left\{v\in\mathcal E:\int_{\R^3}\left(T\left[U_{\xi_i}^2\right]Z_iv
   +2T\left[U_{\xi_i}Z_i\right]U_{\xi_i}v\right)=0\right\},
 \end{equation}
 where
 \begin{equation}\label{T[]}
   T[uv](x)=\frac{1}{4\pi} \int_{\R^3}\frac{u(y)v(y)}{|x-y|}dy
 \end{equation}
 and
\begin{equation*}
  Z_i=\frac{\partial U_{\xi_i}}{\partial r},\quad  i=1,2,\dots,s
\end{equation*}
for simplicity.  Define an auxiliary function
\begin{equation*}
  F(w)=J(U_r+w)\quad w\in\mathcal E_s.
\end{equation*}
 We want to find the critical points of $F$.  To this end, we expand $F(w)$ in the following form:
 \begin{equation}\label{Fw}
   F(w)=F(0)+E(w)+\frac 12   L(w)+N(w), \quad w\in\mathcal E_s,
 \end{equation}
where
\begin{equation}\label{Ew}
   E(w)=\int_{\R^3}(V(|x|)-1)U_rw+\frac12\int_{\R^3}\left(\sum_{i=1}^s
 T\left[U_{\xi_i}^2\right]U_{\xi_i}-T\left[U_r^2\right]U_r\right)w;
\end{equation}
\begin{equation}\label{Lw}
  L(w)=\int_{\R^3}(|\nabla w|^2+V(|x|)w^2)- \frac12\int_{\R^3}
 \left(T\left[U_r^2\right]w^2+2T[U_rw]U_rw\right)
\end{equation}
and
\begin{equation}\label{Nw}
  N(w)=- \frac12\int_{\R^3}T\left[w^2\right]wU_r
  -\frac18\int_{\R^3}T\left[w^2\right]w^2.
\end{equation}

We next will focus on the  terms in the r.h.s. of \eqref{Fw}.  By \eqref{psi-norm}, we know that $E(w)$ is a bounded linear functional in $\mathcal E_s$. In the following we estimate the error $\|E\|_{\mathcal L(\mathcal E_s,\mathcal E_s)}$.

\vskip 2mm
\begin{lemma}\label{error-estimate}
 There exist an integer $s_0 > 0$ and a small $\beta>0$, such that for any $s > s_0$,
 \begin{equation}\label{error}
  \|E\|_{\mathcal L(\mathcal E_s,\mathcal E_s)}\leq C\left(\frac{1}{s}\right)^{\frac{2m-1}{2(1-m)}+\beta}.
 \end{equation}
\end{lemma}
\begin{proof}
 For the first term of $E(w)$, similar arguments as in \eqref{I2} show that
 \begin{equation}\label{Ew1}
   \begin{aligned}
     \int_{\R^3}(V(|x|)-1)U_rw&=s\int_{\R^3}(V(|x-\xi_1|)-1)Uw(x-\xi_1)\\
     &=O\left(\frac{s}{r^m}\right)\|w\|.
   \end{aligned}
 \end{equation}
\par
On the other hand, by \eqref{U1U2} we compute
\begin{equation}\label{Ew2}
  \begin{aligned}
   &\quad \int_{\R^3}\left(\sum_{i=1}^s
 T\left[U_{\xi_i}^2\right]U_{\xi_i}-T\left[U_r^2\right]U_r\right)w\\
 &=-s  \int_{\R^3} T\left[U_{\xi_1}^2 \right] \sum_{i=2}^sU_{\xi_i} w-  \int_{\R^3} T\left[ \sum_{i\neq j}U_{\xi_i}U_{\xi_j}\right]U_r w\\
 &=-s \sum_{i=2}^s \int_{\R^3}\int_{\R^3}
 \frac{U_{\xi_1}^2(x) U_{\xi_i}(y) w(y)}{|x-y|}
 +O\left(\sum_{i\neq j}^se^{-(1-\delta)|\xi_i-\xi_j|}\right)\\
&=-s\sum_{i=2}^sB_i+o\left(e^{-s^{\frac m{1-m}}}\right),
  \end{aligned}
\end{equation}
where
\begin{equation*}
  B_i=\int_{\R^3}\int_{\R^3}
 \frac{U_{\xi_1}^2(x) U_{\xi_i}(y) w(y)}{|x-y|}.
\end{equation*}
\par
To estimate $B_i$, we set
\begin{equation}\label{omega}
  \Omega_j=\left\{x=(x',x_3)\in\R^2\times \R:\left\langle
  \frac{x'}{|x'|},\frac{x_3}{|x_3|}\right\rangle\geq \cos\frac{\pi}{s} \right\},\quad j=1,2,\dots,s.
\end{equation}
 Then we know that for any $x\in\Omega_j$,
 \begin{equation}\label{distance-x}
  |x-\xi_i|\geq \frac12|\xi_j-\xi_i|.
 \end{equation}
Indeed,  it is easy to verify that
  \begin{equation*}
    |x-\xi_j|\leq|x-\xi_i|,
  \end{equation*}
which insures  \eqref{distance-x} if $|x-\xi_j|\ge\frac12|\xi_i-\xi_j|$. Otherwise,
\begin{equation*}
  |x-\xi_i|\geq |\xi_i-\xi_j|-|x-\xi_j|\geq\frac12|\xi_j-\xi_i|.
\end{equation*}
It follows from \eqref{decay} and \eqref{distance-x}  that for any $x\in\Omega_j,$
\begin{equation}\label{Ui-decay}
  U_{\xi_i}(x)\leq
C e^{-(1-\delta)|y-\xi_i|} \leq  C e^{-\frac{1-\delta}{2}|\xi_j-\xi_i|}
\end{equation}
  for any small $\delta>0$.
By \eqref{sum1i}, \eqref{U1U2} and \eqref{Ui-decay}, we find
\begin{equation}\label{Ew2-decay}
  \begin{aligned}
    \sum_{i=2}^sB_i&=\sum_{i=2}^s\int_{\Omega_1}\int_{\R^3} \frac{U_{\xi_1}^2(x) U_{\xi_i}(y) w(y)}{|x-y|}dydx+\sum_{i,j=2}^s\int_{\Omega_j}\int_{\R^3} \frac{U_{\xi_1}^2(x) U_{\xi_i}(y) w(y)}{|x-y|}dydx\\
    &= \sum_{i=2}^s\int_{\Omega_1}\int_{\Omega_i} \frac{U_{\xi_1}^2(x) U_{\xi_i}(y) w(y)}{|x-y|}dydx
    +\sum_{i=2,j=1,\atop i\neq j}^s\int_{\Omega_1}\int_{\Omega_j} \frac{U_{\xi_1}^2(x) U_{\xi_i}(y) w(y)}{|x-y|}dydx\\
    &+ o\left(e^{-s^{\frac m{1-m}}}\right)\\
    &=\sum_{i=2}^s\int_{\Omega_1}\int_{\Omega_i} \frac{U_{\xi_1}^2(x) U_{\xi_i}(y) w(y)}{|x-y|}dydx+ o\left(e^{-s^{\frac m{1-m}}}\right)\\
    &\leq C \sum_{i=2}^s\int_{B_{\frac{|\xi_1-\xi_i|}{8}}(\xi_1)}
    \int_{\Omega_i} e^{-(1-\delta)(2|x-\xi_1|+ |y-\xi_i|)}\frac{ w(y)}{|x-y|}dydx+ o\left(e^{-s^{\frac m{1-m}}}\right)\\
    &\leq C \sum_{i=2}^s\int_{B_{\frac{|\xi_1-\xi_i|}{8}}(\xi_1)\cap \Omega_1}
     \int_{B_{\frac{|\xi_1-\xi_i|}{8}}(\xi_i)\cap \Omega_i} e^{-(1-\delta)(2|x-\xi_1|+ |y-\xi_i|)}\frac{ w(y)}{|x-y|}dydx+ o\left(e^{-s^{\frac m{1-m}}}\right)\\
  &\leq C \sum_{i=2}^s\int_{B_{\frac{|\xi_1-\xi_i|}{8}}(\xi_1)}
     \int_{B_{\frac{|\xi_1-\xi_i|}{8}}(\xi_i)} e^{-(1-\delta)(2|x-\xi_1|+ |y-\xi_i|)}\frac{ w(y)}{|x-y|}dydx+ o\left(e^{-s^{\frac m{1-m}}}\right)\\
  &\leq C \sum_{i=2}^s\frac2{|\xi_1-\xi_i|} \int_{B_{\frac{|\xi_1-\xi_i|}{8}}(\xi_1)}
     \int_{B_{\frac{|\xi_1-\xi_i|}{8}}(\xi_i)} e^{-(1-\delta)(2|x-\xi_1|+ |y-\xi_i|)}  w(y) dydx+ o\left(e^{-s^{\frac m{1-m}}}\right)\\
     &\leq C\frac{s}{r^m} o_r(1)=o\left(\frac{s}{r^m}\right),
  \end{aligned}
\end{equation}
since $w(y)=o_r(1)$ when $|y|\geq\frac r2$.  We conclude from \eqref{Ew1} and \eqref{Ew2} that
\begin{equation*}
  |E(w)|\leq C\frac s{r^m}\|w\|\le C\left(\frac{1}{s}\right)^{\frac{2m-1}{2(1-m)}+\beta}\|w\|
\end{equation*}
because of $\frac12<m<1$. We complete the proof.
\end{proof}

\

With the help of \eqref{Hardy-LS} we easily derive that
\begin{equation*}
  \int_{\R^3}(\nabla \psi_1 \nabla \psi_2+V(|x|)\psi_1 \psi_2)- \frac12\int_{\R^3}
 \left(T\left[U_r^2\right]\psi_1\psi_2+2T[U_r\psi_1]U_r\psi_2\right),\quad \psi_1,\psi_2\in\mathcal E_s
\end{equation*}
 is a bounded bilinear functional $\mathcal E_s$.  So there is a bounded linear operator
$L:\mathcal E_s\rightarrow\mathcal E_s$  such that
\begin{equation*}
  \langle L\psi_1,\psi_2  \rangle=  \int_{\R^3}(\nabla \psi_1 \nabla \psi_2+V(|x|)\psi_1 \psi_2)- \frac12\int_{\R^3}
 \left(T\left[U_r^2\right]\psi_1\psi_2+2T[U_r\psi_1]U_r\psi_2\right),\ \psi_1,\psi_2\in\mathcal E_s.
\end{equation*}
We now consider the invertibility of   the operator $L$ in $\mathcal E_s$.
\begin{lemma}\label{inver}
  There exist a constant $\zeta>0$, independent of $s$,  and an integer $s_0 > 0$  such that for any $s > s_0$,
  \begin{equation*}
    \|Lv\|\ge \zeta\|v\|, \quad v\in\mathcal E_s.
  \end{equation*}
\end{lemma}
 \begin{proof}
   Suppose on the contrary that there exist $s\rightarrow \infty$, $r_s\in I_s$, and $v_s\in\mathcal E_s$  with
   \begin{equation}\label{vs-norm}
     \|v_s\|=\sqrt s \quad \mbox{ and }\quad \|Lv_s\|=o(\|v_s\|).
   \end{equation}
      we first claim that  if $\psi\in \mathcal E$, then $T[U\psi]$ satisfies
    \begin{equation}\label{pro-1}
T[U\psi] \mbox{ is even in   } x_k,\ k=2,3
    \end{equation}
  and
      \begin{equation}\label{pro-2}
        T[U\psi](r\cos\theta,r\sin\theta,x_3)=T[U\psi]\left(r\cos\left(\theta+\frac{2\pi i}{s}\right),r\sin\left(\theta+\frac{2\pi i}{s}\right),x_3\right).
      \end{equation}
 Indeed,   let $h=(h_{ij})\in SO(3)$ and $h=$diag$(1,-1,-1)$.  Then, obviously, $U\in\mathcal E$, which implies  $U(x)=U(hx)$, so is $\psi$. Therefore,
 \begin{equation*}
  T[U\psi](hx)=\int_{\R^3}\frac{U(y)\psi(y)}{|hx-y|}dy=\int_{\R^3}\frac{U(h^{-1}y)\psi(h^{-1}y)}{|x-h^{-1}y|}dy
  =\int_{\R^3}\frac{U(y)\psi(y)}{|x-y|}dy=T[U\psi](x),
 \end{equation*}
  which is equivalent to  \eqref{pro-1}.  A similar argument leads to \eqref{pro-2}.

   In view of the definition of \eqref{omega}, we deduce from \eqref{vs-norm} and symmetry that
   \begin{equation}\label{con-1}
   \begin{aligned}
     \frac{1}{k}\langle Lv_s, u\rangle&=\int_{\Omega_1} (\nabla v_s \nabla u +V(|x|)v_s u)-\int_{\Omega_1}
 \left(T\left[U_r^2\right]v_su+2T[U_ru]U_rv_s\right)\\
 &=o\left(\frac{\|u\|}{\sqrt k}\right),\ \forall u \in\mathcal E_s
   \end{aligned}
   \end{equation}
and
\begin{equation}\label{con-2}
   \int_{\Omega_1} (|\nabla v_s|^2  +V(|x|)v_s^2)=1.
\end{equation}
\par
Since
\begin{equation*}
  Cs^{\frac{m}{1-m}}\le |\xi_2-\xi_1|\le |\xi_i-\xi_1|, \quad i\neq1.
\end{equation*}
 We have that $B_R(\xi_1)\subset\Omega_1$ for  any $R>0$. Let $\tilde v_s(x)=v_s(x+\xi_1)$. Then by \eqref{con-2},
 \begin{equation*}
   \int_{B_R(0)}(|\nabla \tilde v_s|^2  +V(|x|)\tilde v_s^2)\le 1.
 \end{equation*}
 Therefore,  we may assume that there exists a $v\in H^1(\R^3)$, up to a subsequence,  such that as $s\rightarrow\infty$,
\begin{equation*}
  \tilde v_s\rightarrow v, \quad \mbox{ weakly in }H_{loc}^1(\R^3)  \mbox{ and  strogly in }L_{loc}^2(\R^3).
\end{equation*}
 \par
  From $v_s\in \mathcal E_s$ one  derives that
  \begin{equation*}
  \begin{aligned}
       \int_{\R^3}T\left[U_{\xi_1}Z_1\right]U_{\xi_1}v_s&= \int_{\R^3}T\left[U\frac{\partial U}{\partial y_1}\right](x-\xi_1)U_{\xi_1}(x)v_s(x)\\
       &=\int_{\R^3}T\left[U\frac{\partial U}{\partial y_1}\right] U\tilde v_s.
  \end{aligned}
  \end{equation*}
Similarly,
\begin{equation*}
  \int_{\R^3}T\left[U_{\xi_1}^2\right]Z_1v_s=\int_{\R^3}T\left[U^2\right] \frac{\partial U}{\partial x_1}\tilde v_s.
\end{equation*}
Moreover, $\tilde v_s$ is even in $x_k,\ k=2,3$.
Hence,   $v$  is   also even  in $x_k,\ k=2,3$ and  satisfies
\begin{equation}\label{v-con}
  \int_{\R^3}\left(T\left[U^2\right] \frac{\partial U}{\partial x_1}
 +2T\left[U\frac{\partial U}{\partial y_1}\right] U \right)v=0.
\end{equation}
\par
We shall show that $v$ solves
\begin{equation}\label{v-eq}
   -\Delta  v+v=T\left[U^2\right]v
 + 2T\left[U v\right] U.
\end{equation}
Indeed, define the space
\begin{equation*}
  \bar {\mathcal E}=\left\{\phi\in H^1(\R^3):\int_{\R^3}\left(T\left[U^2\right] \frac{\partial U}{\partial x_1}
 +2T\left[U\frac{\partial U}{\partial y_1}\right] U \right)\phi=0\right\}.
\end{equation*}
For any $R>0$, let
\begin{equation*}
  \phi\in\C_0^{\infty}(B_R(0))\cap \bar {\mathcal E} \mbox{  and be even in }x_k,\ k=2,3
\end{equation*}
and denote $\phi_s(x):=\phi(x+\xi_1)$. Then $\phi_s\in\C_0^\infty(B_R(0))$. We may identify  $\phi_s\in\mathcal E_s$   by redefining the values outside $\Omega_1$ with the symmetry.   Following  the
arguments in \cite{WeiCV}, by \eqref{con-1} and   similar arguments in \eqref{Ew2-decay},  we obtain
 \begin{equation}\label{phi1}
    \int_{\R^3} \left(\nabla v \nabla \phi + v \phi -
 T\left[U^2\right]v \phi-2T[Uv]U\phi\right) =0.
 \end{equation}
 \par
 Moreover, since $v$  is  even in $x_k,\ k=2,3$, with the help of \eqref{pro-1},  it is easily shown that       \eqref{phi1} holds for any function $\phi\in \C_0^\infty(B_R(0))$, $\phi$
 odd in $x_k,\ k=2,3$. Therefore,  for any $\phi\in \C_0^\infty(B_R(0))\cap \bar{\mathcal E}$, one gets \eqref{phi1}.  Furthermore,
  \begin{equation}\label{phi2}
        \int_{\R^3} \left(\nabla v \nabla \phi + v \phi -
 T\left[U^2\right]v \phi-2T[Uv]U\phi\right) =0,\quad \forall \phi\in\bar{\mathcal E}
  \end{equation}
 due to the density of $\C_0^\infty(\R^3)$ in $H^1(\R^3)$.
 \par
 On the other hand, \eqref{phi2}   is true for $\phi= \frac{\partial U}{\partial x_1}$. Putting these facts together, we  obtain \eqref{phi2}  for any $\phi\in H^1(\R^3)$.  Namely, \eqref{v-eq} holds.
  \par
  Using the evenness in $x_k,\ k=2,3$ of $v$ again and Theorem \ref{nonden}-(2), we derive that  $  v=c\frac{\partial U}{\partial x_1}.$
By \eqref{v-con}, we further have
\begin{equation*}
  v=0.
\end{equation*}
Hence,
\begin{equation}\label{contra-1}
  \int_{B_R(\xi_1)}v_s^2=o_s(1),\quad \forall R>0.
\end{equation}
\par
It is indicated by \eqref{Ui-decay} that
\begin{equation}\label{contra-2}
  U_r(x)\leq
C e^{-(1-\delta)|x-\xi_1|}, \quad x\in\Omega_1.
\end{equation}
Inserting $u=v_s$ in \eqref{con-1}, we conclude from \eqref{Hardy-LS}, \eqref{contra-1} and \eqref{contra-2} that
\begin{equation*}
  \begin{aligned}
o\left(1\right)&=\int_{\Omega_1} (|\nabla v_s|^2   +V(|x|)v_s^2)-\int_{\Omega_1}
 \left(T\left[U_r^2\right]v_s^2+2T[U_rv_s]U_rv_s\right)\\
&\ge  \int_{\Omega_1} (|\nabla v_s|^2   +V(|x|)v_s^2)+o(1)-C e^{-(1-\delta)R}\int_{\Omega_1}v_s^2\\
 &\ge\frac12\int_{\Omega_1} (|\nabla v_s|^2   +V(|x|)v_s^2)+o(1) \\
 &= \frac12+o(1),\\
  \end{aligned}
\end{equation*}
which is impossible for large  $s$	 and large $R$.   So we complete the proof.
 \end{proof}

\

We now give the existence of  the critical point   for $F$.
\begin{lemma}\label{correction}
  There is an integer $s_0 > 0$  such that for each $s\geq s_0$,  there exists a $\C^1$ map
  \begin{equation*}
    I_s \rightarrow \mathcal E,\ r\mapsto w(r)
  \end{equation*}
  with $w(r)\in\mathcal E_s$ and
  \begin{equation*}
    \langle \frac{\partial F(w)}{\partial w},\phi\rangle=0,\quad\forall \phi\in\mathcal E_s.
  \end{equation*}
  Moreover, there is a small constant $\beta>0$ such that
  \begin{equation}\label{w-estimate}
 \|w\|\le    C\left(\frac{1}{s}\right)^{\frac{2m-1}{2(1-m)}+\beta}.
  \end{equation}
\end{lemma}
\begin{proof}
   It follows from  Riesz Theorem that there is a $e_s \in \mathcal E_s$ satisfying
\begin{equation*}
  E(w)=(e_s,w) \mbox { and } \|e_s\|=  \|E\|_{\mathcal L(\mathcal E_s,\mathcal E_s)}.
\end{equation*}
   Consequently, $w$, a critical point of $F$ in $\mathcal E_s$, solves
   \begin{equation}\label{w-eq}
      e_s+\mathcal L w+N'(w)=0.
   \end{equation}
   It follows from  Lemma \ref{inver} that $\mathcal L$ is invertible. Furthermore,  we can write \eqref{w-eq} as follows
   \begin{equation*}
     w=P(w):=-\mathcal L^{-1}\left(e_s+N'(w)\right).
   \end{equation*}
Let
\begin{equation*}
   M:=\left\{\varphi\in \mathcal E_s : \|\varphi\|\le  C\left(\frac{1}{s}\right)^{\frac{2m-1}{2(1-m)}+\beta} \right\}.
\end{equation*}
We claim that $P$ is a contraction map from $M$ to $M$.   Indeed,  in view of \eqref{Nw},  we readily get
\begin{equation*}
\|N'(\varphi)\|\leq C\|\varphi\|^2 \mbox{ and }\|N''(\varphi)\|\leq C\|\varphi\|.
\end{equation*}
We derive from Lemmas \ref{error-estimate}--\ref{inver} that
 \begin{equation*}
 \begin{aligned}
     \|P(\varphi)\|&\le C\|e_s\|+C\|\varphi\|^2  \le C\left(\frac{1}{s}\right)^{\frac{2m-1}{2(1-m)}+\beta}.
 \end{aligned}
 \end{equation*}
Furthermore,
 \begin{equation*}
 \begin{aligned}
     \|P(\varphi_1)-P(\varphi_2)\|&\le  C\|N'(\varphi_1)-N'(\varphi_2)\|
       \le C\left(\frac{1}{s}\right)^{\frac{2m-1}{2(1-m)}+\beta}
      \|\varphi_1-\varphi_2\|\\
      &\le \frac12 \|\varphi_1-\varphi_2\|
 \end{aligned}
 \end{equation*}
 for $s$ large.
    The proof is complete.
\end{proof}

\

Define
\begin{equation*}
  \mathcal F(r)=E(U_r+w),\quad \forall r\in I_s,
\end{equation*}
where   $I_s$ is
  defined in  \eqref{r-range}.
It is  easily checked  that for $s$ sufficiently large,
if $ r$ is a critical point of $\mathcal F$, then  $ U_{  r}+w$ is a solution of \eqref{SN1}, see for example \cite{Rey} or \cite{Kang, Lin07}.

\begin{proof}[Proof of Theorem \ref{main2}]
    We conclude from \eqref{Fw} and Lemmas \ref{energy-estimate},  \ref{error-estimate} and  \ref{correction} that
    \begin{equation}\label{Fr}
      \begin{aligned}
        \mathcal F(r)&=E(U_r)+E(w)+\frac 12   \langle Lw,w\rangle +N(w)\\
        &=E(U_r)+O( \|E\|_{\mathcal L(\mathcal E_s,\mathcal E_s)})\\
        &= s\left[\frac{A_2}{16\pi}+\frac{aA_1}{2} \frac{1}{r^m}- \frac{A_1^2}{128\pi^2}\frac{s\log s}{r}+ O\left( \frac1{s^{\frac{m}{1-m}+\beta}} \right) \right]
      \end{aligned}
    \end{equation}
 Consider the following maximization problem
   \begin{equation}\label{max}
      \max_{r\in I_r}\mathcal F(r).
   \end{equation}
It is easy to check that   the function
\begin{equation}\label{g}
  g(r)=\frac{aA_1}{2} \frac{1}{r^m}- \frac{A_1^2}{128\pi^2}\frac{s\log s}{r}
\end{equation}
has a maximum point
\begin{equation}\label{rs}
  r_s=\left(\left( \frac{A_1}{64am\pi^2}\right)^{\frac{1}{1-m}}+o(1)\right)(s\log s)^{\frac{1}{1-m}}.
\end{equation}
It follows from the expression of   $\mathcal F(r)$ that the maximizer $  r_s$	 is an interior
point of  $I_r$.     Thus,
 system \eqref{SN} admits  a positive  solution $ u_s=U_{r_s}+w_s$ with the required properties.
\end{proof}

We conclude this section with the following remark about the assumption $(H)$ on the potential $V$.
\begin{remark} \label{rem}
Observe that when $0<m<\frac 12$,    we deduce from Lemmas \ref{energy-estimate},  \ref{error-estimate} and  \ref{correction} that
    \begin{equation*}
      \begin{aligned}
        \mathcal F(r)&=E(U_r)+E(w)+\frac 12   \langle Lw,w\rangle +N(w)\\
        &=E(U_r)+O( \|E\|_{\mathcal L(\mathcal E_s,\mathcal E_s)})\\
        &= s\left[\frac{A_2}{16\pi}+\frac{aA_1}{2} \frac{1}{r^m}- \frac{A_1^2}{128\pi^2}\frac{s\log s}{r} \right]+c\frac{s^2}{r^{2m}}.
      \end{aligned}
    \end{equation*}
  If we consider the function
  \begin{equation*}
    \bar g(r)=- \frac{A_1^2}{128\pi^2}\frac{s\log s}{r}+c \frac s {r^{2m}},
  \end{equation*}
  a direct calculation shows  that  $\bar g$ has a maximum point
\begin{equation}\label{bars}
  \bar r_s= C(\log s)^{\frac 1m}.
\end{equation}
     Note that
    \begin{equation*}
      \lim_{r\rightarrow\infty}\frac{1/r^m}{s/r^{2m}}= \lim_{r\rightarrow\infty}\frac{r^m}{s}=0
    \end{equation*}
and one has to center the spikes on a circle of radius
 \begin{equation*}
   r\in[(C-\alpha)(\log s)^{\frac 1m},(C+\alpha)(\log s)^{\frac 1m}].
 \end{equation*}
  However,   the distance between two spikes is
  \begin{equation}\label{d12}
    |\xi_1-\xi_2|=2r\sin\frac{\pi}{s}\sim \frac{2\pi r}{s}\rightarrow 0, \mbox{ as }
    s \rightarrow \infty.
  \end{equation}
  Therefore, it seems challenging to construct $s-$bump solutions for \eqref{SN} in this situation.
  A similar phenomenon happens for $m\ge 1$.
\end{remark}

\

\


\vskip6mm




 \begin{thebibliography}{}



\bibitem{WeiDcds}
W. Ao, J. Wei, W. Yang, Infinitely many positive solutions of fractional nonlinear Schr\"{o}dinger equations with non-symmetric potentials.   {\em Discrete Contin. Dyn. Syst. A} 37 (2017) 5561-5601.



\bibitem{Chen-wy}
W. Chen, J. Wei, S. Yan, Infinitely many solutions for the Schr\"{o}dinger equations in $\R^N$ with critical growth. {\em  J. Differential Equations} 252 (2012) 2425-2447

\bibitem{Dsp}
P, d'Avenia,  A. Pomponio,
 G. Vaira,
Infinitely many positive solutions for a Schr\"{o}dinger-Poisson system.
{\em Nonlinear Anal. } 74 (2011) 5705-5721.

\bibitem{erdos}
L. Erd\"{o}s, H.T. Yau, Derivation of the nonlinear Schr\"{o}dinger equation from a
many-body Coulomb system. \emph{Adv. Theor. Math. Phys.} 5 (2001) 1169-1205.

\bibitem{Trudinger1997}
 D. Gilbarg, N. Trudinger,
Elliptic Partial Differential Equations of Second Order,
     \newblock Springer-Verlag, Berlin Heidelberg New York, 1997.

    \bibitem{Luo2}
     Q. Guo, P. Luo, C. Wang, J. Yang,
     Existence and local uniqueness of normalized peak solutions for a Schr\"{o}odinger-Newton system. 	arXiv:2008.01557.

\bibitem{hepp}
K. Hepp, The classical limit for quantum mechanical correlation functions, \emph{Commun. Math. Phys.} 35 (1974) 265-277.

\bibitem{Kang}
 X. Kang, J. Wei,   On interacting bumps of semi-classical states on nonlinear Schr\"{o}dinger equations. {\em Adv. Differ.
Equ. } 5 (2000) 899-928.


\bibitem{Li10}
G. Li, S. Peng, S. Yan, Infinitely many positive solutions for the nonlinear Schr\"{o}dinger-Poisson system. {\em  Commun.
Contemp. Math. }12 (2010) 1069-1092.

\bibitem{lieb}
E. Lieb, Existence and uniqueness of the minimizing solution of Choquard's nonlinear equation. \emph{Studies
in Appl. Math.} 57 (1976/77), 93-105.

\bibitem{Lin07} 
F. H. Lin, W. M. Ni, J. Wei, On the number of interior peak solutions for a singularly perturbed Neumann
problem. \emph{Comm. Pure Appl. Math.} 60 (2007) 252-281.

\bibitem{Lions0}
P. L. Lions, The Choquard equation and related questions. \emph{Nonlinear Anal.} 4 (1980) 1063-1072.

\bibitem{Lions}
P. L. Lions, The concentration-compactness principle in the calculus of variations. The limit case, Part 1. {\em Rev. Mat.
Iberoam. } 1 (1985) 145-201.

\bibitem{Lions2}
P. L. Lions, The concentration-compactness principle in the calculus of variations. The limit case, Part 2. {\em Rev. Mat. Iberoam. }  1 (1985) 45-121.

\bibitem{Luo}
P. Luo, S. Peng, C. Wang,
Uniqueness of positive solutions with concentration for the
Schr\"{o}dinger-Newton  problem.  {\em Calc. Var. Partial Differ. Equ. } 59 (2020)  60.

\bibitem{menzala}
G. P. Menzala, On regular solutions of a nonlinear equation of Choquard's type. \emph{Proc. Roy. Soc. Edinburgh
Sect. A} 86 (1980) 291-301.

\bibitem{Tod99}
 I. M.  Moroz, P. Tod,  An analytical approach to the Schr\"{o}dinger-Newton equations. {\em Nonlinearity } 12  (1999)
201-216.
\bibitem{Tod98}
 I. M.  Moroz, R. Penrose, P. Tod,    Spherically-symmetric solutions of the  Schr\"{o}dinger-Newton equations. {\em Class. Quantrum Grav. } 15 (1998) 2733-2742.

 \bibitem{WeiP}
M. Musso, J. Wei, S. Yan, Infinitely many positive solutions for an nonlinear field equation with super-critical growth. {\em Proc. London Math. Soc. } 112 (2016) 1-26.

\bibitem{Peng1}
 S. Peng, Z.-Q. Wang, Segregated and synchronized vector solutions for nonlinear Schr\"{o}dinger systems. {\em
Arch. Rational Mech. Anal. } 208 (2013) 305-339.

\bibitem{Peng2}
 S. Peng, Q. Wang,  Z.-Q. Wang, On coupled nonlinear Schr\"{o}dinger systems with mixed couplings.
 {\em Trans. Amer. Math. Soc. } 371 (2019) 7559-7583.

\bibitem{Penrose}
R. Penrose,   Quantum computation, entanglement and state reduction. {\em R. Soc. Lond. Philos. Trans. Ser.
A Math. Phys. Eng. Sci. } 356 (1998) 1927-1939.

\bibitem{Rey}
 O. Rey, The role of the Green's function in a nonlinear elliptic equation involving
the critical Sobolev exponent. {\em J. Funct. Anal. } 89 (1990) 1-52.

\bibitem{spohn}
H. Spohn, Kinetic equations from Hamiltonian dynamics. \emph{Rev. Mod. Phys.}
52 (1980) 569-615.

\bibitem{Wei09}
 L. Wang, J. Wei, S. Yan, A Neumann problem with critical exponent in
non-convex domain and Lin-Ni's conjecture. {\em Trans. Amer. Math. Soc. }   362 (2010) 4581-4615.


\bibitem{Wang13}
 L. P. Wang, C. Y. Zhao, Infinitely many solutions for the prescribed boundary mean curvature
problem in $\R^N$.  {\em Canad. J. Math. } 65 (2013) 927-960.

\bibitem{WeiJmp}
J.  Wei,  M. Winter,   Strongly interacting bumps for the Schr\"{o}dinger-Newton equations. {\em  J. Math. Phys. }
50   (2009) 012905



\bibitem{WeiJfa}
  J. Wei, S. Yan, Infinitely many solutions for the prescribed scalar curvature
problem on $\s^n$. {\em  J. Funct. Anal. } 258 (2010) 3048-3081.

 \bibitem{WeiAsna}
   J. Wei, S. Yan, An elliptic problem with critical growth and Lazer-Mckenna
conjecture.  {\em Ann. Scuola. Norm. Pisa. } 9 (2010)  423-457.



  \bibitem{WeiJmpa}
 J. Wei, S. Yan,  Infinitely many positive solutions for an elliptic problem with critical or supercritical
growth. {\em J. Math. Pures Appl. } 96 (2011) 307-333.

\bibitem{WeiCV}
 J. Wei, S. Yan, Infinitely many solutions for the nonlinear Schr\"{o}dinger equations
in $\R^N$.   {\em Calc. Var. Partial Differ. Equ.} 37 (2010) 423-439.



 \end{thebibliography}
\end{document}